\documentclass[11pt,a4paper,english]{article}   
\usepackage{amsthm, amsmath,amssymb}    
\usepackage{babel}
\usepackage[latin1]{inputenc}
\usepackage[T1]{fontenc}
\usepackage{enumerate}

\theoremstyle{plain}
 \newtheorem{theorem}{Theorem}[section]

 \newtheorem{corollary}[theorem]{Corollary}
 
 \newtheorem{problem}[theorem]{Problem}
\theoremstyle{definition}
 
 \newtheorem{example}[theorem]{Example}

\theoremstyle{remark}
 \newtheorem{remark}[theorem]{Remark}

\def\Ker{\mathrm{Ker}}
\def\exp{\mathrm{exp}}
\def\Exp{\mathrm{Exp}}

\usepackage{epsfig}
 \usepackage{color}

\newcommand{\titulo}{Multivariate interpolation}

\newcommand{\authorA}{Pascual Jara}
\newcommand{\direcA}{Department of Algebra. University of Granada. SPAIN}
\newcommand{\authorB}{Joaquín Jódar}
\newcommand{\direcB}{Department of Mathematics. University of Jaén. SPAIN}
\newcommand{\authorC}{Luis Merino}
\newcommand{\direcC}{Department of Algebra. University of Granada. SPAIN}
\newcommand{\authorD}{Juan F. Ruiz}
\newcommand{\direcD}{Department of Mathematics. University of Jaén. SPAIN}

\begin{document}

\title{\titulo}
\author{\authorA\\\authorB\\\authorC\\\authorD}

\maketitle

\begin{abstract}
\noindent

The aim of this work is to show how symbolic computation can be used to perform multivariate Lagrange, Hermite and Birkhoff interpolation and help us to build more realistic interpolating functions. After a theoretical introduction in which we analyze the complexity of the method we shall focus our attention on applications.
\end{abstract}

\section*{Introduction}

Multivariate interpolation consists in given finitely many points $p_1,\ldots,p_t$ in a $n$--dimensional affine space, over a field $K$, such that $p_i\neq{p_j}$ if $i\neq{j}$, and $t$ scalars $v_1,\ldots,v_t$, determining a polynomial $F\in{K[X_1,\ldots,X_n]}$ such that $F(p_i)=v_i$ for any $i=1,\ldots,t$. This is known as the \emph{Lagrange interpolation problem}. It is well known that there are several methods and algorithms to compute such a polynomial $F$.

In this paper we will add some extra conditions and show some algorithms to find an interpolating polynomial. We will deal with possible values of the interpolating polynomials at the given points, as in the Lagrange problem, and also with values of some higher order derivatives at these points. We consider two different cases. The first deals with considering unitary vectors $U$ in $K^n$ and values for all derivatives $D_U^jF(p_i)$ for $j=1,\ldots,s(i)$, the so called \emph{Hermite interpolation problem} consists in determining a polynomial $F\in{K[X_1,\ldots,X_n]}$ satisfying all these conditions, see~\cite{Bojanov/Xu:2002, Bojanov/Xu:2007} and~\cite{Gasca/Sauer:2000}. The second deals with the same situation, but we assume some gaps may exist, i.e., there may exists some point $p_i$, some order $j\leq{s_i}$ or some vector $U$ such that we have no assigned value for $D_U^jF(p_i)$; the so called \emph{Birkhoff interpolation problem} consists in determining a polynomial $F\in{K[X_1,\ldots,X_n]}$ satisfying these conditions.

In two of the three cases when we ``\emph{homogenize}'' the relationships, i.e., when we
consider $D_U^j(p_i)=0$ and $F(p_i)=0$ for any $U$, $j$ and $p_i$, the set of all
polynomials satisfying all these relationships constitutes an ideal of the polynomial
ring $K[X_1,\ldots,X_n]$. This allows us to use some computational method in
$K[X_1,\ldots,X_n]$, at least when $K$ is a computable field, in order to determine a
general solution to the interpolation problem. This method is based on the application
of Groebner basis. On this matter we study some results about the complexity of the
algorithm and the best monomial order we need to use. The third case, the Birkhoff
interpolation problem, does not produce ideals in the polynomial ring, hence the
application of Groebner basis is not allowed. But we may reorganize the information to
work as in the previous two cases to get a general solution of the interpolation
problem, and after that, to particularize to get an specific one. Nevertheless in this
case every gap will introduce a parameter to the specific solution, hence we will obtain
the Birkhoff interpolation polynomial as a solution to an indeterminate compatible
system of linear equations. In the practice, a high number of parameters (or
equivalently gaps) slows down the computation of the interpolation polynomial.

One of the theoretical problems we solve in this paper is to determine the existence and uniqueness of the Hermite interpolation polynomial, which in the Lagrange case is well known. We obtain this result as a consequence of the study of solutions to the Birkhoff interpolation problem.

Let us describe the content of each section in this paper.

In Section one we recall some facts about multivariate Lagrange interpolation, and in Section two we introduce the multivariate Hermite interpolation by using Groebner basis theory. In particular we prove that if we consider elements $p_1,\ldots,p_t$ in $\mathbb{A}^n(K)$ such that $p_i\neq{p_j}$ if $i\neq{j}$, unitary vectors $U_{i,j_i}\in{K^n}$, $i=1\ldots,t$, $j_i=1\ldots,s(i)$, and sets $H_i\subseteq\mathbb{N}^{s(i)}$ such that if $h\in{H}$ and $h-e\in\mathbb{N}^{s(i)}$, then $h-e\in{H_i}$, and define for any $h=(h_1,\ldots,h_{s(i)})\in{H_i}$ and any polynomial $F\in{K[X_1,\ldots,X_n]}$ the derivative $D^{(h)}F=D_{U_{i,1}}^{(h_1)}\cdots{D_{U_{i,s(i)}}^{(h_{s(i)})}}F$, then for any set $\{v_{i,h}\in{K}\mid\;h\in{H_i}\}$ there is a polynomial $F\in{K[X_1,\ldots,X_n]}$ satisfying $D^{(h)}F(p_i)=v_{i,h}$ for any $p_i$ and any $h\in{H_i}$. This polynomial $F$ is unique modulo the cofinite ideal $\{F\in{K[X_1,\ldots,X_n]}\mid\;D^{(h)}F(p_i)=0,\mbox{ for all }p_i\mbox{ and }h\in{H_i}\}$.

In Section three we consider the more intricate case in which some gaps appear, i.e., when the sets $H_i$ have some gaps. In general the set $\{F\in{K[X_1,\ldots,X_n]}\mid\;D^{(h)}F(p_i)=0,\mbox{ for all }p_i\mbox{ and }h\in{H_i}\}$ is not an ideal, but we may extend the set of conditions adding those corresponding to the gaps. In this case we have new sets, $\overline{H_i}$, in such a way that the set $\{F\in{K[X_1,\ldots,X_n]}\mid\;D^{(h)}F(p_i)=0,\mbox{ for all }p_i\mbox{ and }h\in\overline{H_i}\}$ is a cofinite ideal. Let $\mathbb{G}$ be a Groebner basis, then a $K$--vector space basis of the quotient ring is parameterized by $\mathbb{N}^n\setminus\Exp(\mathbb{G})$, and this produces a determinate compatible system of linear equations: $\{D^{(h)}F(p_i)=0\}_{i,h\in\overline{H_i}}$, and an indeterminate compatible system of linear equations $\{D^{(h)}F(p_i)=v_{i,h}\}_{i,h\in{H_i}}$, where the rank of this system depends on the number of gaps.

Section four is devoted to exhibit the algorithms we develop and Section five contains some examples of how these algorithms work. In particular we are interested in showing that for a given real function the addition of conditions on higher order derivatives produces a better approximation.

\section{Multivariate Lagrange interpolation}

The problem of multivariate interpolation consists in given $t$ different elements,
$p_1,\ldots,p_t\in\mathbb{A}^n(K)$ in the affine $n$--dimensional space, and  $t$ elements $v_1,\ldots,v_t\in{K}$,
determining a polynomial $F\in{K[X_1,\ldots,X_n]}$ such that $F(p_i)=v_i$ for every
$i=1,\ldots,t$.

\bigskip\noindent
The following algorithm allows us to determine a polynomial $F$ satisfying this property;
see~\cite{deBoor,Pistone/Wynn:1996}. For any index $i$ we consider a linear operator
$L_i:K[X_1,\ldots,X_n]\longrightarrow{K}$ defined as:
\[
 L_i(G)=G(p_i),\mbox{ for every }G\in{K[X_1,\ldots,X_n]}.
\]
The kernel of $L_i$ is $\Ker(L_i)=\{G\in{K[X_1,\ldots,X_n]}\mid\;G(p_i)=0\}$, i.e.,
$\Ker(L_i)=(X_1-p_{i,1},\ldots,X_n-p_{i,n})$, where $p_i=(p_{i,1},\ldots,p_{i,n})$. Thus
$\Ker(L_i)$ is exactly the ideal of the point $p_i$.

The ideal $I=\cap_{i=1}^t\Ker(L_i)=\mathcal{I}(\{p_1,\ldots,p_t\})$, is the ideal of the
finite set $\{p_1,\ldots,p_t\}$. As a consequence the quotient ring $K[X_1,\ldots,X_n]/I$ is
a finite dimensional vector space over the field $K$. Therefore there exists a finite
basis of $K[X_1,\ldots,X_n]/I$.

There are different methods to compute effectively this basis. Indeed, we consider a monomial order in $\mathbb{N}^n$, hence every non--zero polynomial $F\in{K[X_1,\ldots,X_n]}$ can be written uniquely as $\sum_{\alpha\in{A}\in\mathbb{N}^n}k_\alpha{X^\alpha}$, where $A$ is a finite set, $0\neq{k_\alpha}\in{K}$ and $X^\alpha=X_1^{\alpha_1}\cdots{X_n^{\alpha_n}}$ whenever $\alpha=(\alpha_1,\ldots,\alpha_n)\in\mathbb{N}^n$.
We call $\exp(F)$ the maximum in $A$, and define $\Exp(I)=\{\exp(F)\in\mathbb{N}^n\mid\;F\in{I}\}$. A Groebner basis if $I$ is a set $\{G_1,\ldots,G_s\}\subseteq{I}$ such that $\Exp(I)=\{\exp(G_1),\ldots,\exp(G_s)\}+\mathbb{N}^n$. The existence of a Groebner basis for each ideal $I\subseteq{K[X_1,\ldots,X_n]}$ is well known as it is the uniqueness of a reduced Groebner basis. See~\cite{COX}.
The quotient ring $K[X_1,\ldots,X_n]/I$ has finite dimension over $K$ if and only if the set $\mathcal{B}=\mathbb{N}^n\setminus\Exp(I)$ is finite, and it can be taken as an index set for a vector space basis of $K[X_1,\ldots,X_n]/I$. Indeed, the set $\{X^\beta+I\mid\;\beta\in\mathcal{B}\}$ is a basis.

For any polynomial
$F$, which is a solution of the interpolation problem, i.e., $F(p_i)=v_i$ for any
index $i=1,\ldots,t$, we may write
$$
F+I=\sum_{\beta\in\mathcal{B}}k_\beta{X^\beta}+I\mbox{ for some }k_\beta\in{K}.
$$
Let $F_0$ be the polynomial $\sum_{\beta\in\mathcal{B}}k_\beta{X^\beta}$, then
$F(p_i)=F_0(p_i)$, for any index $i=1,\ldots,t$. Hence $F_0$ is an interpolating
polynomial which we may compute by solving the linear equation system in the unknowns $\{k_\beta\}_{\beta\in\mathcal{B}}$:
\[
\left.
\sum_{\beta\in\mathcal{B}}k_\beta{p_i^\beta}=v_i,\;{i=1,\ldots,t}
\right\}
\]
Observe that the uniqueness of the interpolating polynomial is determined modulus the ideal $I$ as if $F_1,F_2$ are interpolating polynomials for any index $i=1,\ldots,t$ we have $F_1(p_i)=v_i=F_2(p_i)$, hence $F_1-F_2\in{I}$.
We remark that different monomial orders give, in general, different interpolating polynomials.

In order to study the existence of the interpolating polynomial we address to the next sections in which we study the multivariate Hermite and Birkhoff interpolation problem.

\section{Multivariate Hermite interpolation}\label{se:two}

In addition to the conditions in the Lagrange interpolation problem, it is of interest, sometimes, to impose some extra conditions in order to put more information in the interpolating polynomial. Derivatives is one of the tools we shall use. Intuitively, in the univariate case, if we have two different points $p_1$ and $p_2$, two values
$v_1$ and $v_2$, and we want to compute an interpolating polynomial $F$ such that $F(p_i)=v_i$ we may proceed in the usual way.
But if the point $p_2$ tends to $p_1$, we would like to use the derivative in $p_1$ instead $p_2$. Thus the univariate Hermite interpolation deals with the problem of computing an interpolating polynomial
$F$ such that given different elements $p_1,\ldots,p_t\in{K}$, and elements $v_{0,1},\ldots,v_{0,t},v_{1,1},\ldots,v_{1,t}\in{K}$ the following relations hold:
\[
\left.
\begin{array}{ll}
 F(p_i)=v_{0,i},\\
 F'(p_i)=v_{1,i}
\end{array}
\right\}
\qquad\mbox{ for every }i=1,\ldots,t.
\]
We may also consider derivatives of higher order. In this case we need to impose some extra
condition in order to assure we deal with ideals in the polynomial ring. See remarks
below.

\vskip.5cm
Let us describe the multivariate Hermite interpolation problem.

Let $p_1,\ldots,p_t\in\mathbb{A}^n(K)$, such that $p_i\neq{p_j}$ if $i\neq{j}$,
linearly independent unitary vectors $U_{i,j_i}\in{K^n}$, for $i=1,\ldots,t$, $j_i=1,\ldots,s(i)$, and
elements $v_1,\ldots,v_t\in{K}$, $v_{i,j_i,h}\in{K}$, $i=1,\ldots,t$, $j_i=1,\ldots,s(i)$, $h=1,\ldots,s(i,j_i)$.
The multivariate Hermite interpolation problem consists in determining a polynomial
$F\in{K[X_1,\ldots,X_n]}$ such that:
\begin{equation}\label{eq:0}
\left.
\begin{array}{ll}
 F(p_i)=v_i,\mbox{ for every }i=1,\ldots,t;\mbox{ and}\\\\
 D^{(h)}_{U_{i,j_i}}F(p_i)=v_{i,j_i,h},\mbox{ for every }i=1,\ldots,t,\;j_i=1,\ldots,s(i),\;h=1,\ldots,s(i,j_i)
\end{array}
\right\}
\end{equation}
In order to determine such a polynomial $F$ let us first consider the problem of determining all polynomials $F\in{K[X_1,\ldots,X_n]}$ such that:
\begin{equation}\label{eq:05}
\left.
\begin{array}{ll}
 F(p_i)=0,\mbox{ for every }i=1,\ldots,t;\mbox{ and}\\\\
 D^{(h)}_{U_{i,j_i}}F(p_i)=0,\mbox{ for every }i=1,\ldots,t,\;j_i=1,\ldots,s(i),\;h=1,\ldots,s(i,j_i)
\end{array}
\right\}
\end{equation}

To study these polynomials we consider a single point, say $p_i$, we assume $p_i=0$, and the system
\begin{equation}\label{eq:1}
\left.
\begin{array}{l}
 F(p_i)=0,\\\\
 D^{(h)}_{U_{i,j_i}}F(p_i)=0,\mbox{ for every }j_i=1,\ldots,s(i),\;h=1,\ldots,s(i,j_i)
\end{array}
\right\}
\end{equation}
The sum of two roots of this system also is, and for any root $F$ and any $G\in{K[X_1,\ldots,X_n]}$ we have:
\[
(FG)(p_i)=F(p_i)G(p_i)=0G(p_i)=0.
\]
\[
D^{(h)}_{U_{i,j_i}}(FG)(p_i)=\sum_{k=0}^hD^{(h-k)}_{U_{i,j_i}}F(p_i)D^{(k)}_{U_{i,j_i}}G(p_i)=0.
\]
Thus the set of all roots of system~\eqref{eq:1} constitutes an ideal, say $I_i$. We claim the ideal $I_i$ is co--finite. In fact, $I_i$ contains the ideal $J_i=\langle{X_1^{e_1}\cdots{X_n^{e_n}}\mid\;e_1+\cdots+e_n=s(i,j_i)+1}\rangle$.

Now we may study the set of roots of system~\eqref{eq:05}, i.e., the intersection $I=I_1\cap\ldots\cap{I_t}$. Hence it is an ideal of $K[X_1,\ldots,X_n]$. We have that $I$ is co--finite as $I$ is an intersection of finitely many co--finite ideals.

If $\mathbb{G}$ is a reduced Groebner basis of $I$ and $\Exp(I)\subseteq\mathbb{N}^n$ is the mono-ideal of all exponents of elements in $I$, a basis of the vector space $K[X_1,\ldots,X_n]/I$ is parameterized by $\mathcal{B}=\mathbb{N}^n\setminus\Exp(I)$. Indeed, if we consider the basis $\{X^\beta\mid\;\beta\in\mathcal{B}\}$, a generic element in $K[X_1,\ldots,X_n]/I$ has a representative in the form
\[
 \sum_{\beta\in\mathcal{B}}k_\beta{X^\beta}
\]
for some $k_\beta\in{K}$. Every condition in~\eqref{eq:0} produces a linear equation with unknowns in
$\{k_\beta\mid\;\beta\in\mathcal{B}\}$. Thus we obtain a system of linear equations:
\begin{equation}\label{eq:2}
\left.
\begin{array}{ll}
 (\sum_{\beta\in\mathcal{B}}k_\beta{X^\beta})(p_i)=v_i,\mbox{ for every }i=1,\ldots,t;\mbox{ and}
 \\\\
 D^{(h)}_{U_{i,j_i}}(\sum_{\beta\in\mathcal{B}}k_\beta{X^\beta})(p_i)=v_{i,j_i,h},\;\mbox{for every}\,i=1,\ldots,t,\,j_i=1,\ldots,s(i),\,h=1,\ldots,s(i,j_i)
\end{array}
\right\}
\end{equation}
Any solution of system~\eqref{eq:2} gives an interpolating polynomial.

Observe that if this solution exists, it is unique modulo the ideal $I$.

\subsection{Generalized multivariate Hermite interpolation}\label{subsec:21}

We may also consider a more general version of the Hermite interpolation problem. Indeed,
we consider elements $p_1,\ldots,p_t\in\mathbb{A}^n(K)$ such that $p_i\neq{p_j}$ if
$i\neq{j}$, linearly independent unitary vectors $U_{i,j_i}\in{K^n}$, $i=1,\ldots,t$, $j_i=1,\ldots,s(i)$ and
sets $H_i\subseteq\mathbb{N}^{s(i)}$ such that if $h\in{H_i}$, and
$h-e\in\mathbb{N}^{s(i)}$, then $h-e\in{H_i}$ for every $e\in\mathbb{N}^{s(i)}$.

For any $h=(h_1,\ldots,h_{s(i)})\in{H_i}$ and any polynomial $F\in{K[X_1,\ldots,X_n]}$ we define
\[
D^{(h)}F=D^{(h_1)}_{U_{i,1}}\cdots{D^{(h_{s(i)})}_{U_{i,s(i)}}F}.
\]
With this notation, the generalized multivariate Hermite problem consists in, given
$\{v_{i,h}\mid\;h\in{H_i}\}\subseteq{K}$, determining a
polynomial $F\in{K[X_1,\ldots,X_n]}$ such that:
\begin{equation}\label{eq:3}
\left.
\begin{array}{l}
 D^{(h)}F(p_i)=v_{i,h},\mbox{ for every }i=1,\ldots,t,\;h\in{H_i}
\end{array}
\right\}
\end{equation}

Observe that for any index $i=1,\ldots,t$ we may prove, \emph{mutatis--mutandi}, that the set of solutions of
\begin{equation}\label{eq:4}
\left.
\begin{array}{l}
 D^{(h)}F(p_i)=0,\mbox{ for every }i=1,\ldots,t,\;h\in{H_i}
\end{array}
\right\}
\end{equation}
is an ideal, say $I_i$. Hence the algorithm may be built as we exemplified above. This
will be the algorithm we develop in Section~\ref{se:algo}.

Let us resume in the following theorem these results.

\begin{theorem}\label{th}
With the notation in this section there exists a unique interpolating polynomial $F$,
modulo the ideal $I$, such that
\[
\left.
D^{(h)}F(p_i)=v_{i,h},\mbox{ for every }i=1,\ldots,t;\;h\in{H_i}
\right\}
\]
\end{theorem}
\begin{proof}
We only need to prove the existence of such an interpolating polynomial $F$; it is a direct consequence of Corollary~\eqref{co:090304d}.
\end{proof}

\section{Multivariate Birkhoff interpolation}\label{se:bir}

The multivariate Hermite interpolation method does not work if there exist some derivative gaps, as in this case we
can not assure $I_i$ is an ideal. This problem may be solved if we fill these gaps in order to compute the ideal associated to each point $p_i$. Let us perform this process and show how to compute this new ideal.

As in subsection~\eqref{subsec:21}, let us consider points $p_1,\ldots,p_t$ such that $p_i\neq{p_j}$ if $i\neq{j}$, linearly independent unitary vectors $U_{i,j_i}\in{K^n}$, $i=1,\ldots,t$, $j_i=1,\ldots,s(i)$ and sets $H_i\subseteq\mathbb{N}^{s(i)}$. As usual, for any index $i$, any $h\in{H_i}$ and any polynomial $F\in{K[X_1,\ldots,X_n]}$ we define
\[
D^{(h)}F=D^{(h_1)}_{U_{i,1}}\cdots{D^{(h_{s(i)})}_{U_{i,s(i)}}F}.
\]
The multivariate Birkhoff interpolation problem consists in given $\{v_{i,h}\mid\;h\in{H_i}\}\subseteq{K}$ determining a polynomial $F\in{K[X_1,\ldots,X_n]}$ such that
\[
\left.
D^{(h)}F(p_i)=v_{i,h},\mbox{ for every }i=1,\ldots,t,\;h\in{H_i}
\right\}
\]

If we fix an index $i$, we are interested in the set of all polynomials
\[
\left\{F\in{K[X_1,\ldots,X_n]}\mid\;D^{(h)}F(p_i)=0,\;\mbox{ for all }h\in{H_i}\right\}.
\]
Contrary to the theory developed in subsection~\eqref{subsec:21}, this set of polynomials is not an ideal. To remedy this hitch we proceed as follows. First we define $b_i$ the maximum in the set
\[
\left\{|h|:=\sum_{l=1}^{s(i)}h_l\mid\;h=(h_l)_l\in{H_i}\right\},
\]
and define $\overline{H_i}=\{h\in\mathbb{N}^{s(i)},\mbox{ such that }|h|\leq{b_i}\}$. Second we consider the polynomial equations
\begin{equation}\label{eq:090304}
\left.
D^{(h)}F(p_i)=0,\mbox{ for every }h\in\overline{H_i}
\right\}
\end{equation}

Let $J_i$ the set of all polynomials satisfying these equations. We obtain that $J_i$ is an ideal of $K[X_1,\ldots,X_n]$.

Similar arguments, to those developed in section~\eqref{se:two}, give us that the ideal $J_i$ is cofinite. Hence if we define $J=J_1\cap\ldots\cap{J_t}$, then $J$ is a cofinite ideal of $K[X_1,\ldots,X_n]$.

Let $\mathbb{G}$ be a Groebner basis of $J$ and $\mathcal{B}$ the complement in $\mathbb{N}^n$ of $\Exp(J)$, then $\mathcal{B}$ has finitely many elements and $\{X^\beta\mid\;\beta\in\mathcal{B}\}$ is a vector space basis of the quotient $K[X_1,\ldots,X_n]/J$. So every element in $K[X_1,\ldots,X_n]/J$ has a unique representation of the shape $\sum_{\beta\in\mathcal{B}}k_\beta{X^\beta}$, where $k_\beta\in{K}$.

In order to compute an interpolating polynomial $F$ which is a solution of the multivariate Birkhoff interpolating problem we need to solve the linear equation system
\[
\left.
D^{(h)}\left(\sum_{\beta\in\mathcal{B}}k_\beta{X^\beta}\right)(p_i)=v_{i,h},\;i=1,\ldots,t,\;h\in{H_i}
\right\}
\]

\begin{theorem}\label{th:existence}
With the notation in this section there exists an interpolating polynomial $F$ such that
\[
\left.
D^{(h)}F(p_i)=v_{i,h},\;\mbox{ for every }i=1,\ldots,t;\;h\in{H_i}
\right\}
\]
\end{theorem}
\begin{proof}
For any index $i=1,\ldots,t$ we consider the ideal $J_i$. This ideal can be computed in an easy way. Indeed, if we assume $p_i=0$, then $J_i$ is the ideal generated by all monomials $X^\varepsilon$ such that $|\varepsilon|\geq{b_i}+1$. In particular, a minimal set of generators of $J_i$ is
\[
\left\{X^\varepsilon\mid\;|\varepsilon|=b_i+1\right\}.
\]
The codimension of $J_i$ is finite, say $d_i$.

We claim the codimension of $J=J_1\cap\ldots\cap{J_t}$ is $d:=d_1+\cdots+d_t$.

First we show that $J_i+J_j=K[X_1,\ldots,X_n]$ whenever $i\neq{j}$. Indeed, in this case there is a component, say 1, such that $p_{i,1}\neq{p_{j,1}}$. Since
$(X_1-p_{i,1})^{b_i}\in{J_i}$,
$(X_1-p_{j,1})^{b_j}\in{J_j}$ and
$X_1-p_{i,1}$, $X_1-p_{j,1}$ are coprime,
then $J_i+J_j=K[X_1,\ldots,X_n]$.
We claim $(J_1\cap{J_2})+J_3=K[X_1,\ldots,X_n]$. Indeed, there are $x_{1,3}\in{J_1}$, $x_{2,3}\in{J_2}$ and $x_{3,1},x_{3,2}\in{J_3}$ such that $x_{1,3}+x_{3,1}=1=x_{2,3}+x_{3,2}$, then $1=x_{1,3}x_{2,3}+(x_{1,3}x_{3,2}+x_{3,1}x_{2,3}+x_{3,1}x_{3,2})\in{(J_1\cap{J_2})+J_3}$. By induction we obtain $(J_1\cap\ldots\cap{J_{i-1}})+J_i=K[X_1,\ldots,X_n]$ for any $2\leq{i}\leq{t}$.

Now the result follows, by induction on $i$, from the following isomorphisms:
\[
\begin{array}{ll}
\begin{displaystyle}
\frac{K[X_1,\ldots,X_n]/(J_1\cap\ldots\cap{J_i})}{J_i/(J_1\cap\ldots\cap{J_i})}
\cong\frac{K[X_1,\ldots,X_n]}{J_i},
\end{displaystyle}\\\\
\begin{displaystyle}
\frac{J_i}{(J_1\cap\ldots\cap{J_i})}
\cong\frac{J_1\cap\ldots\cap{J_{i-1}}+J_i}{J_1\cap\ldots\cap{J_{i-1}}}
\cong\frac{K[X_1,\ldots,X_n]}{J_1\cap\ldots\cap{J_{i-1}}}.
\end{displaystyle}
\end{array}
\]
Observe that $\#(\mathcal{B})=d=d_1+\cdots+d_t$.

We consider the linear system
\begin{equation}\label{eq:090304b}
\left.
D^{(h)}\left(\sum_{\beta\in\mathcal{B}}k_\beta{X^\beta}\right)(p_i)=0,
\mbox{ for every }i=1,\ldots,t,\;h\in\overline{H_i}
\right\}
\end{equation}
This system has $d$ equations as it has $d_i$ equations for every index $i$. The matrix of this system is a square matrix. Since the system has unique solution, this matrix is regular.

Therefore if we consider the linear system
\begin{equation}\label{eq:090304c}
\left.
D^{(h)}\left(\sum_{\beta\in\mathcal{B}}k_\beta{X^\beta}\right)(p_i)=v_{i,h},
\mbox{ for every }i=1,\ldots,t,\;h\in{H_i}
\right\}
\end{equation}
it is an indeterminate compatible linear system. In particular there is at least a solution. Hence there is an interpolating polynomial which is a solution to the multivariate Birkhoff interpolation problem..
\end{proof}

We remark that the solution to the Birkhoff problem is not unique, in general if there are some derivative gaps.

As a consequence of this Theorem we have the following result which completes the proof of Theorem~\eqref{th}

\begin{corollary}\label{co:090304d}
There is an interpolating polynomial solving the multivariate Hermite interpolation problem.
\end{corollary}

\subsection{Final comments}

\subsubsection*{Increasing the number of restrictions}
In the three multivariate interpolation methods we have studied if we add a new point, it is not necessary, in order to compute an interpolating polynomial, to start again the algorithm from the very beginning. Let us show this situation in the case of the multivariate Birkhoff interpolation method.
Let $p_{t+1}$ be a new point, different from $p_1,\dots,p_t$,
$U_{t+1,1},\ldots,U_{t+1,s(t+1)}$ unitary vectors in $K^n$,
$H_{t+1}\subseteq\mathbb{N}^{s(t+1)}$  and
elements $\{v_{t+1,h}\in{K}\mid\;h\in{H_{t+1}}\}$.
To compute the new interpolating polynomial we only need to compute the following elements:
\begin{enumerate}[(1)]
\item
the new ideal $J_{t+1}$ by computing a Groebner basis $\mathbb{G}_{t+1}$;
\item
the intersection $J\cap{J_{t+1}}$ by computing a Groebner basis $\mathbb{G}$, using the Groebner basis of $J$ and $\mathbb{G}_{t+1}$;
\item
a basis of $K[X_1,\ldots,X_n]/(J\cap{J_{t+1}})$ using $\mathbb{N}^n\setminus\Exp(\mathbb{G})$ and
\item
solve the linear system
\[
\left.
D^{(h)}F(p_i)=v_{i,h},\mbox{ for every }i=1,\ldots,t+1;\;h\in{H_i}
\right\}
\]
\end{enumerate}

We proceed in the same way when we add some extra restriction on high degree in the
derivatives at the points $p_1,\ldots,p_t$.

\subsubsection*{Different orders in $\mathbb{N}^n$ produce a different shape of the interpolating polynomial}

It is of interest to observe that different monomial orders in $\mathbb{N}^n$ give
different shape of the interpolation polynomial. For instance, the lexicographic
monomial order with the ordering of unknowns $X_1>X_2>\cdots>X_n$ produces that the
interpolating polynomial has greater degree in the greatest labeling unknowns.
Otherwise, the graded (reverse) lexicographic order produces interpolating polynomials
in which all unknowns have similar degree. The reason is that the shape of the set
$\mathbb{N}^n\setminus\Exp(\mathbb{G})$ strongly depends on the chosen monomial order.

\subsubsection*{Complexity of the algorithm}

There are different factors in order to perform a quick algorithm.

First we need to compute a Groebner basis for each ideal $I_i$ or $J_i$, depending of
the methods. This process is well established; the speed depends of the kind of numbers
we use in the coordinates of points. For integer numbers it works in an ``acceptable''
way and for other kind of numbers, rational, reals, complexes, single or multiple
precision, etc., it is more and more slow.

The second step is computing a Groebner basis of $I$ or $J$. We observe the best method
is to make it as follows: $I_1\cap{I_2}$, $I_1\cap{I_2}\cap{I_3}$, \ldots The reason is
that anyway we need to perform a total of $t-1$ intersections of ideals and the
described method realizes at the $i$-th steep the intersection
$(I_1\cap\cdots\cap{I_i})\cap{I_{i+1}}$, where at least $I_{i+1}$ has a simple
description.

The third step is to determine a general interpolating polynomial using
$\mathbb{N}^n\setminus\Exp(\mathbb{G})$ as set of indices, and after that to solve a
linear system. It is at the last point where the type of data is again determinant in
order to get quickly a solution. We also may solve numerically this system in order to
get an approximation to the final solution.

\subsubsection*{Existence and uniqueness of the interpolating polynomial}

As we mentioned earlier different monomial orders in $\mathbb{N}^n$ produce different
interpolation polynomials, and we also proved in Theorem~\eqref{th:existence} the
existence of a solution to the multivariate Birkhoff interpolation problem. Otherwise
the uniqueness of an interpolating polynomial depends on the restrictions we have, but in
any case it can be established only modulo the ideal $I$ or $J$, depending of the
method. In Theorem~\eqref{th}, see also Corollary~\eqref{co:090304d}, the uniqueness was
established for the multivariate Hermite interpolation problem. In the Birkhoff
interpolation problem uniqueness fails as we show in the section devoted to examples at
the end of this paper; the reason is the existence of some derivative gaps at some
point.

\section{Multivariate Interpolation Algorithm}\label{se:algo}

We describe the algorithm for the Hermite interpolation problem. With the same notation
that in subsection~\ref{subsec:21}, let $p_1,\ldots,p_t\in\mathbb{A}^n(K)$, such that
$p_i\neq{p_j}$ if $i\neq{j}$, unitary vectors $U_{i,j_i}\in{K^n}$, for $i=1,\ldots,t$,
$j_i=1,\ldots,s(i)$, and elements
$\{v_{i,h}\mid\;h\in{H_i},\;i=1,\ldots,t\}\subseteq{K}$.

The algorithm has three blocks.

(I) The first one determines a Groebner basis of the ideals $I_i$ of all roots of
system~\eqref{eq:4}, 

(II) The second one determines a Groebner basis of $I=I_1\cap\ldots\cap{I_t}$.

(III) The third one gives the interpolated polynomial $F$ satisfying conditions in
\eqref{eq:3}.

We now describe explicitly the developed algorithm for computing a Hermite interpolating
polynomial.

\subsection{Block I}

Compute a Groebner Basis of every ideal $I_i$, for any index $i=1,\ldots,t$.


\begin{tabular}{l}
LET $p:=p_i=(x_1,\ldots,x_n)\in\mathbb{A}^n(K)$ \\
LET $r:=b_i$; $b_i=\max\{|h|:=\sum_{j=1}^{s(i)}h_{j}\mid\;h=(h_{j})_{j}\in{H_{i}}\}$\\
LET $s:=s(i)$\\
LET $H:=H_i$ \\
LET $GEN(I):=\varnothing$\\
IF $r=0$ THEN\\
\hspace*{0.5cm}$GEN(I)=\{X_1-x_1,\ldots,X_n-x_n\}$\\
ELSE \\
\hspace*{0.5cm}FOR $l=1$ TO $r$\\
\hspace*{1.0cm}LET $\alpha(l):=\{\alpha\in\mathbb{N}^n\mid\;|\alpha|:=\alpha_1+\cdots+\alpha_n=l\}$\\
\hspace*{1.0cm}LET $P:=\sum_{|\alpha|=l}A(\alpha)X_1^{\alpha_1}\cdots{X_n^{\alpha_n}}$\\
\hspace*{1.0cm}LET $Der_l:=\{h\in{H}\mid\;|h|=l\}$\\
\hspace*{1.0cm}LET $m:=\#(Der_l)$\\
\hspace*{1.0cm}FOR $k=1$ TO $m$\\
\hspace*{1.5cm}LET $h(k)\in{Der_l}$\\
\hspace*{1.5cm}LET $P_k:=D^{h(k)}P=\sum_{|\alpha|=l}c_k(\alpha)A(\alpha)$\\
\hspace*{1.0cm}ENDFOR\\
\hspace*{1.0cm}LET $f:K^{\#(\alpha(l))}\longrightarrow{K^m}$,
    $f((A(\alpha))_{\alpha\in\alpha(l)})=(\sum_{|\alpha|=l}c_k(\alpha)A(\alpha))_{k=1}^m$\\
\hspace*{1.0cm}COMPUTE a basis of the kernel of the linear map $f$\\
\hspace*{1.0cm}FOR $q=1$ TO $dim(Ker(f))$\\
\hspace*{1.5cm}LET $d_q(\alpha)$ be the $q$-th element in the basis of $Ker(f)$\\
\hspace*{1.5cm}ADD TO $GEN(I)$ the element
    $\sum_{|\alpha|=l}d_q(\alpha)(X_1-x_1)^{\alpha_1}\cdots(X_n-x_n)^{\alpha_n}$\\
\hspace*{1.0cm}ENDFOR\\
\hspace*{0.5cm}ENDFOR\\
\hspace*{0.5cm}ADD TO $GEN(I)$ all the monomials of degree $r+1$ in $X_1-x_1,\ldots,X_n-x_n$\\
ENDIF\\
COMPUTE a Groebner basis $\mathbb{B}$ of $\langle{GEN(I)}\rangle$\\
LET $I_i=\langle\mathbb{B}\rangle$.
\end{tabular}

\subsection{Block II}

Compute $I=I_1\cap\ldots\cap{I_t}$ giving a Groebner basis.

\vskip.5cm

\begin{tabular}{l}
LET $\mathbb{I}:=\{I_1,\ldots,I_t\}$\\
LET $\mathcal{O}$ the graded reverse lexicographic order on $X=\{X_1,\ldots,X_n\}$\\
LET $\mathcal{P}$ the lexicographic order on $\{X_0,X\}$ related to $\mathcal{O}$,\\
\hspace*{1.5cm}where $X_0$ is the auxiliary variable to realize the intersection through the elimination method\\
LET $I=I_1$\\
FOR $k=2$ TO $t$\\
\hspace*{0.5cm}LET $I=$INTERSECTION$(I,I_k,\mathcal{P})$\\
ENDFOR\\
\end{tabular}

\vskip.5cm

INTERSECTION$(I,J,\mathcal{P})$ provides the ideal intersection of $I$ and $J$ giving a Groebner basis of it with respect to the
order $\mathcal{P}$.

\subsection{Block III}

Let us to show the algorithm to find the set $\mathbb{N}^n\setminus\Exp(I)$. After that
we only need to establish the system of linear equations in \eqref{eq:3} and solve it to
find the interpolating polynomial.

\vskip.5cm

\begin{tabular}{l}
LET $I:=\langle{g_1,\ldots,g_k}\rangle$, where $\{g_1,\ldots,g_k\}$ is the Groebner basis obtained in Block II\\
LET $(\alpha_{t,1},\ldots,\alpha_{t,n}):=$ The leading exponent of $g_t$\\
LET $m_i:=max\{\alpha_{1,i},\ldots,\alpha_{k,i}\}$\\
LET $p:=(m_1-1,\ldots,m_n-1)$\\
COMPUTE $\exp:=\{h\in\mathbb{N}^n\mid p-h\in\mathbb{N}^n\}$\\
LET $b:=\varnothing$\\
\hspace*{0.5cm}FOR i=1 TO \#(exp)\\
\hspace*{0.5cm}LET $bt:=\emptyset$\\
\hspace*{0.5cm}LET $j:=0$\\
\hspace*{1.0cm}WHILE $(j<k$ AND $bt=\varnothing)$ \\
\hspace*{1.5cm}$j++$\\
\hspace*{1.5cm}IF $(exp_{i,s}>=\alpha_{j,s}, \forall s=1,..,n)$ THEN \\
\hspace*{2.0cm}$bt=\{exp_i\}$\\
\hspace*{1.5cm}ENDIF\\
\hspace*{1.0cm}ENDWHILE\\
\hspace*{0.5cm}$b=b \cup bt$\\
\hspace*{0.5cm}ENDFOR\\
LET $\exp:=\exp-b$
\end{tabular}

\vskip.5cm

Therefore we take
\[
F=\sum_{\alpha\in{\exp}}A(\alpha)X^\alpha.
\]

\begin{remark}
The Birkhoff interpolation problem admits a similar algorithm in which blocks II and III
are applied in the same way and for block I we have that, in the notation used in this
paper, a minimal set of generators of the ideal $J_i$ is directly given by
$$
\{(X_1-p_{i,1})^{e_1}\ldots (X_n-p_{i,n})^{e_n}\mid\;e_1+\cdots+e_n=b_{i}+1\}.
$$
\end{remark}

\section{Examples}

\begin{example}
Let us show an example in which the use of derivatives of higher order
gives a more exact approximation to the surface we are studying. We begin with a
ellipsoid $E$ of equation
\[
F(X,Y)=(cos(X)cos(Y),cos(X)sin(Y),3sin(X))
\]
and fix eight points in $E$, for instance
\[
p_1=(2,0),\,
p_2=(2,\pi/2),\,
p_3=(2,\pi),\,
p_4=(2,3\pi/2),
\]
\[
p_5=(3,\pi/4),\,
p_6=(3,3\pi/4),\,
p_7=(3,5\pi/4)\mbox{ and }
p_8=(3,7\pi/4).
\]
\end{example}

We have the following picture:

\begin{center}
\begin{tabular}{lll}
\includegraphics[height=.3\textheight]{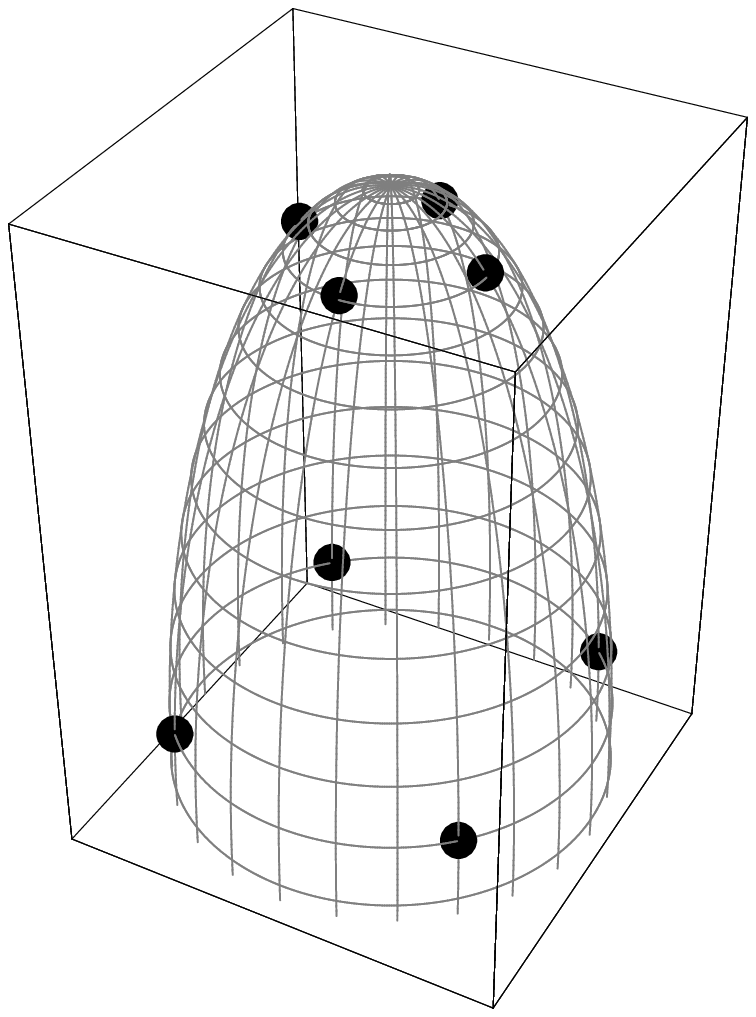}&
\hspace*{2cm}&
\includegraphics[height=.3\textheight]{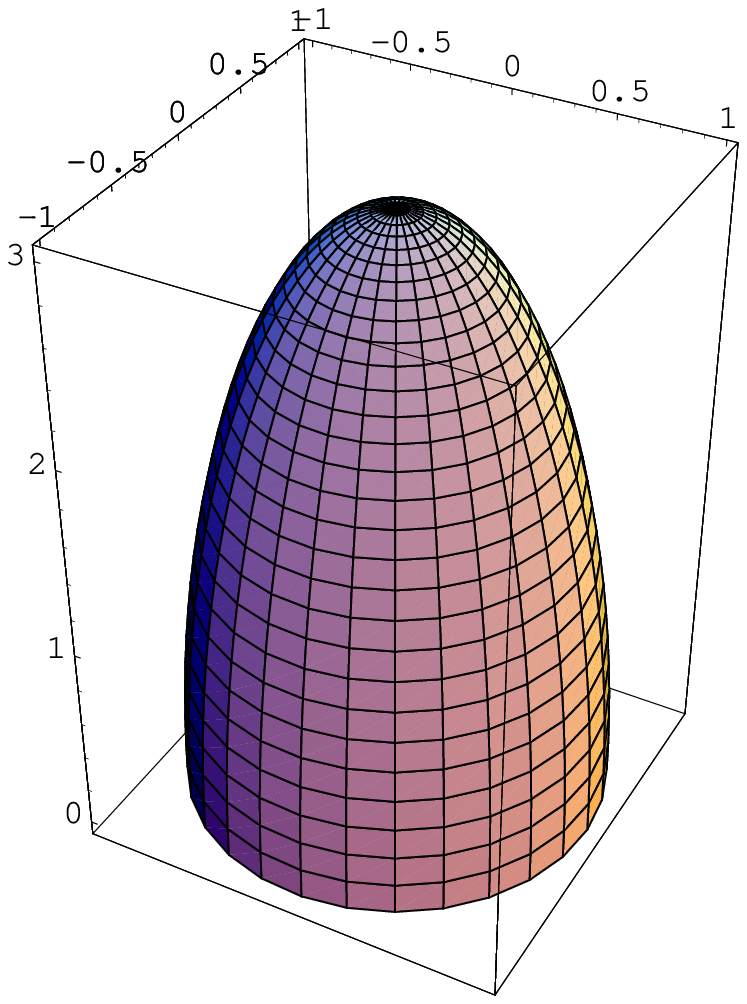}\\
\end{tabular}
\end{center}

We may apply multivariate Lagrange interpolation with the following values
\[
\begin{array}{|l|l|}\hline
p_1=(2,0),\,            &v_1=(cos(2),0,3sin(2))\\
p_2=(2,\pi/2),\,        &v_2=(0,cos(2),3sin(2))\\
p_3=(2,\pi),\,          &v_3=(-cos(2),0,2sin(2))\\
p_4=(2,3\pi/2),\,       &v_4=(0,-cos(2),3sin(2))\\
p_5=(3,\pi/4),\,        &v_5=(cos(3)/\sqrt{2},cos(3)/\sqrt{2},3sin(3))\\
p_6=(3,3\pi/4),\,       &v_6=(-cos(3)/\sqrt{2},cos(3)/\sqrt{2},3sin(3))\\
p_7=(3,5\pi/4),\,       &v_7=(-cos(3)/\sqrt{2},-cos(3)/\sqrt{2},3sin(3))\\
p_8=(3,7\pi/4),\,.      &v_8=(cos(3)/\sqrt{2},-cos(3)/\sqrt{2},3sin(3))\\\hline
\end{array}
\]

and represent this polynomial. Thus we obtain:
\begin{center}
\begin{tabular}{lll}
\includegraphics[height=.3\textheight]{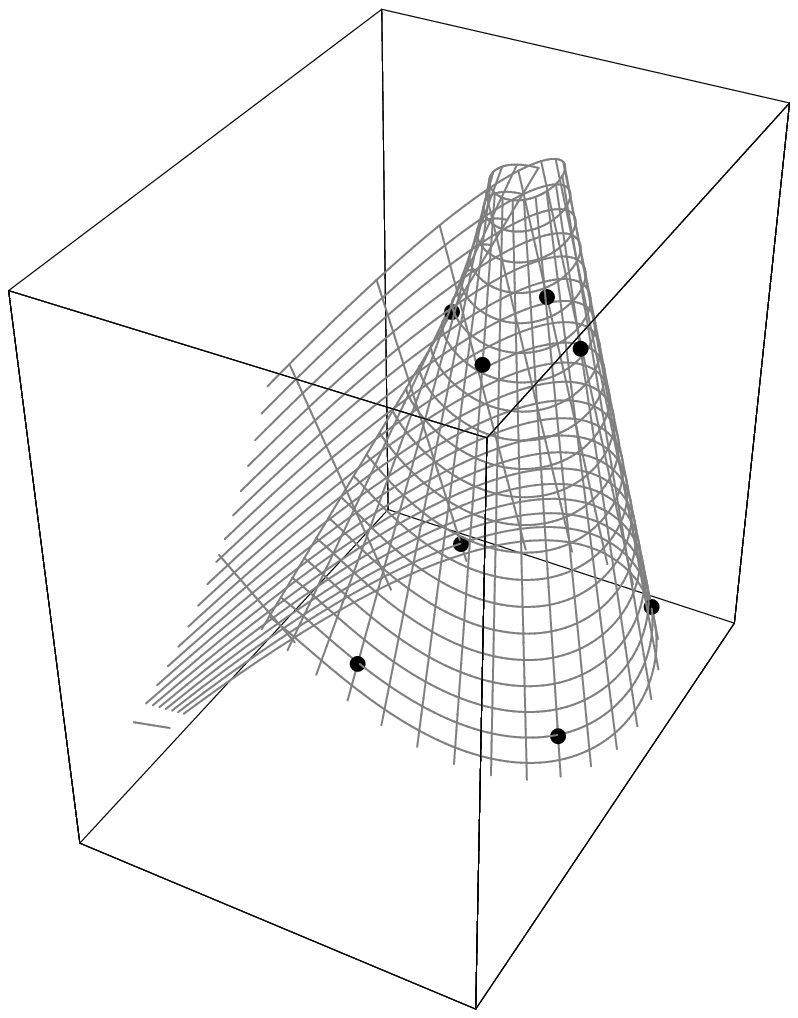}&
\hspace*{2cm}&
\includegraphics[height=.3\textheight]{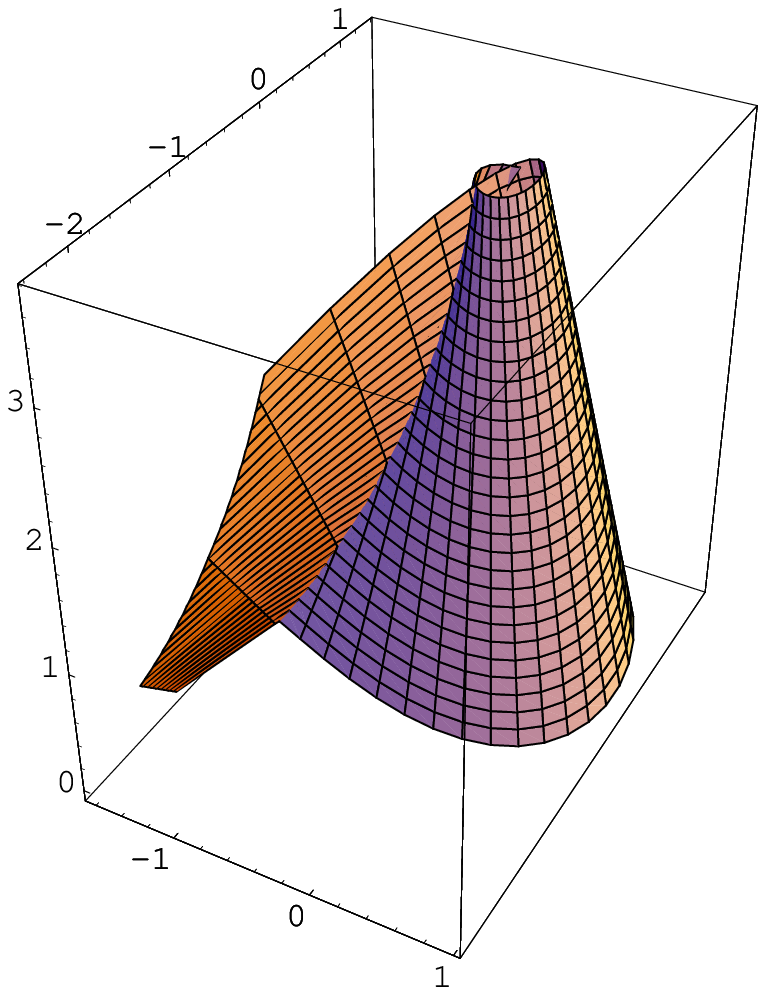}\\
\end{tabular}
\end{center}

Now we try to compute the multivariate Hermite polynomial with derivatives up to order one. The table of data is:
\begin{equation}\label{ta:090305}
\begin{array}{|l|l|l|l|}\hline
 p_i&F(p_i)&D_{(1,0)}F(p_i)&D_{(0,1)}F(p_i)\\\hline
 (2,0) & (\cos (2),0,3 \sin (2)) & (-\sin (2),0,3 \cos (2))
   & (0,\cos (2),0) \\
 \left(2,\frac{\pi }{2}\right) & (0,\cos (2),3 \sin (2)) &
   (0,-\sin (2),3 \cos (2)) & (-\cos (2),0,0) \\
 (2,\pi ) & (-\cos (2),0,3 \sin (2)) & (\sin (2),0,3 \cos
   (2)) & (0,-\cos (2),0) \\
 \left(2,\frac{3 \pi }{2}\right) & (0,-\cos (2),3 \sin (2)) &
   (0,\sin (2),3 \cos (2)) & (\cos (2),0,0) \\
 \left(3,\frac{\pi }{4}\right) & \left(\frac{\cos
   (3)}{\sqrt{2}},\frac{\cos (3)}{\sqrt{2}},3 \sin (3)\right) &
   \left(-\frac{\sin (3)}{\sqrt{2}},-\frac{\sin (3)}{\sqrt{2}},3
   \cos (3)\right) & \left(-\frac{\cos (3)}{\sqrt{2}},\frac{\cos
   (3)}{\sqrt{2}},0\right) \\
 \left(3,\frac{3 \pi }{4}\right) & \left(-\frac{\cos
   (3)}{\sqrt{2}},\frac{\cos (3)}{\sqrt{2}},3 \sin (3)\right) &
   \left(\frac{\sin (3)}{\sqrt{2}},-\frac{\sin (3)}{\sqrt{2}},3
   \cos (3)\right) & \left(-\frac{\cos
   (3)}{\sqrt{2}},-\frac{\cos (3)}{\sqrt{2}},0\right) \\
 \left(3,\frac{5 \pi }{4}\right) & \left(-\frac{\cos
   (3)}{\sqrt{2}},-\frac{\cos (3)}{\sqrt{2}},3 \sin (3)\right) &
   \left(\frac{\sin (3)}{\sqrt{2}},\frac{\sin (3)}{\sqrt{2}},3
   \cos (3)\right) & \left(\frac{\cos (3)}{\sqrt{2}},-\frac{\cos
   (3)}{\sqrt{2}},0\right) \\
 \left(3,\frac{7 \pi }{4}\right) & \left(\frac{\cos
   (3)}{\sqrt{2}},-\frac{\cos (3)}{\sqrt{2}},3 \sin (3)\right) &
   \left(-\frac{\sin (3)}{\sqrt{2}},\frac{\sin (3)}{\sqrt{2}},3
   \cos (3)\right) & \left(\frac{\cos (3)}{\sqrt{2}},\frac{\cos
   (3)}{\sqrt{2}},0\right)\\\hline
\end{array}
\end{equation}

If we represent the Hermite interpolation polynomial we obtain:
\begin{center}
\begin{tabular}{lll}
\includegraphics[height=.3\textheight]{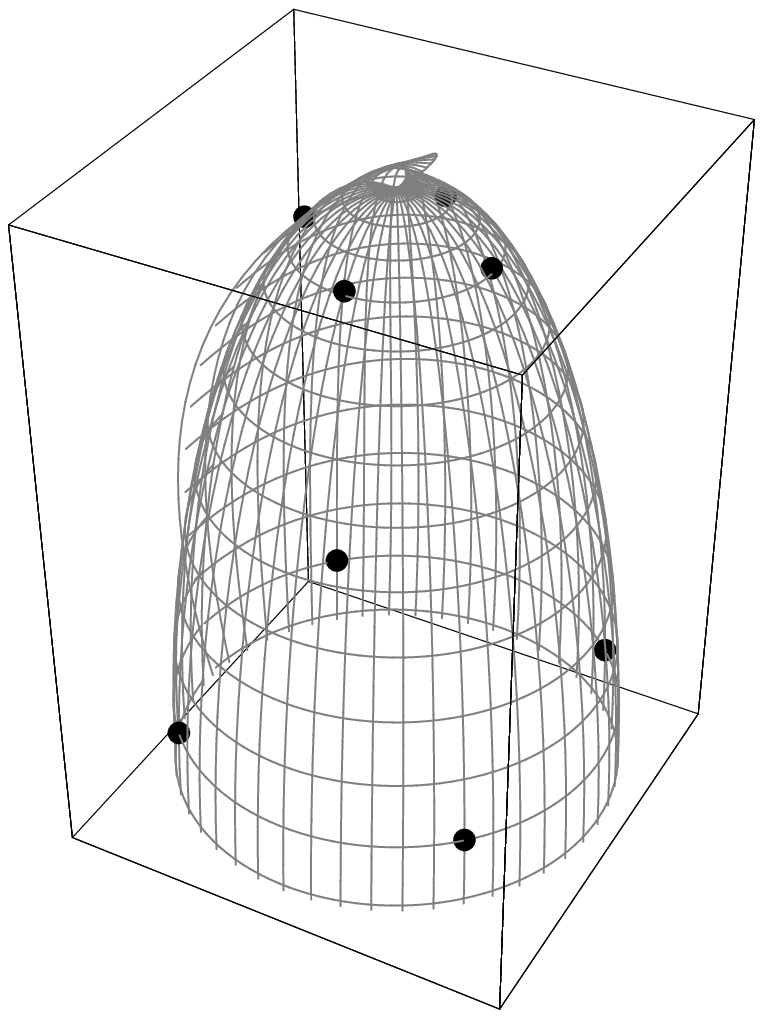}&
\hspace*{2cm}&
\includegraphics[height=.3\textheight]{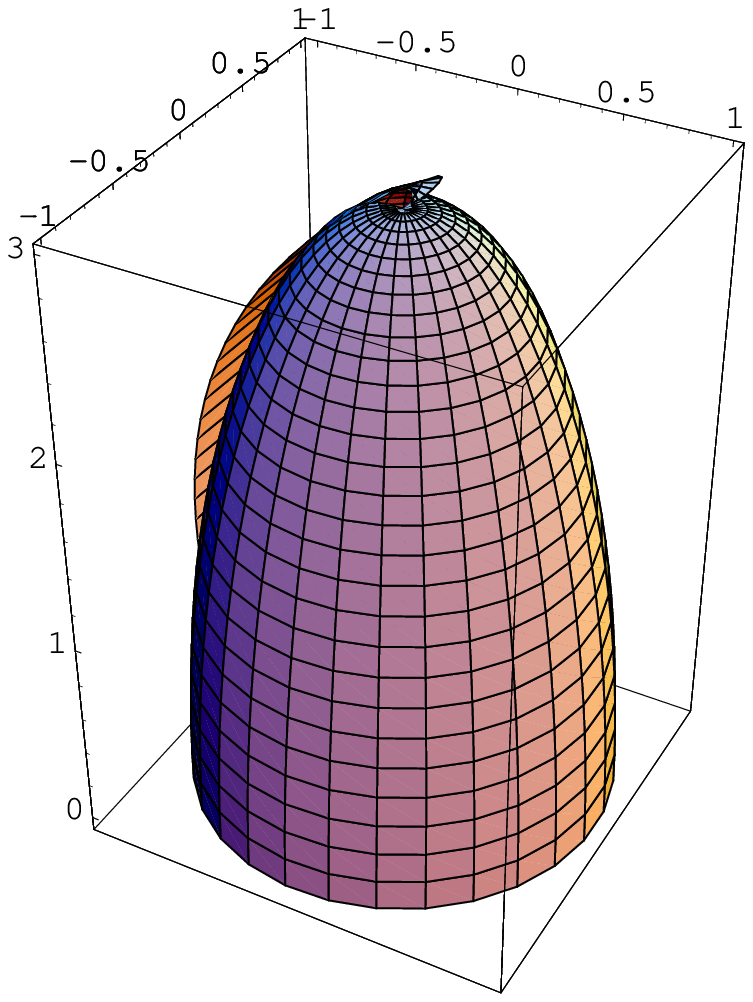}\\
\end{tabular}
\end{center}

Observe that in this case the new surface is a better approximation to the original one.

As a third guest we compute the multivariate Hermite polynomial with derivative up to order two. In this case we need to add new columns to the table given in~\eqref{ta:090305}.

\begin{equation}\label{eq:090305b}
\begin{array}{|l|l|l|l|l|}\hline
p_i&&D_{(1,0)}D_{(1,0)}F(p_i)&D_{(1,0)}D_{(0,1)}F(p_i)&D_{(0,1)}D_{(0,1)}F(p_i)\\\hline
(2,0)
& \cdots
& (-\cos (2),0,-3 \sin (2))
& (0,-\sin (2),0)
& (-\cos (2),0,0) \\
\left(2,\frac{\pi }{2}\right)
& \cdots
& (0,-\cos  (2),-3 \sin (2))
& (\sin (2),0,0)
& (0,-\cos (2),0) \\
(2,\pi )
& \cdots
& (\cos (2),0,-3 \sin (2))
& (0,\sin (2),0)
& (\cos (2),0,0) \\
\left(2,\frac{3 \pi }{2}\right)
& \cdots
& (0,\cos (2),-3\sin (2))
& (-\sin (2),0,0)
& (0,\cos (2),0) \\
\left(3,\frac{\pi }{4}\right)
& \cdots
& \left(-\frac{\cos(3)}{\sqrt{2}},-\frac{\cos (3)}{\sqrt{2}},-3 \sin (3)\right)
& \left(\frac{\sin (3)}{\sqrt{2}},-\frac{\sin(3)}{\sqrt{2}},0\right)
& \left(-\frac{\cos(3)}{\sqrt{2}},-\frac{\cos (3)}{\sqrt{2}},0\right) \\
\left(3,\frac{3 \pi }{4}\right)
& \cdots
& \left(\frac{\cos (3)}{\sqrt{2}},-\frac{\cos (3)}{\sqrt{2}},-3\sin (3)\right)
& \left(\frac{\sin (3)}{\sqrt{2}},\frac{\sin(3)}{\sqrt{2}},0\right)
& \left(\frac{\cos(3)}{\sqrt{2}},-\frac{\cos (3)}{\sqrt{2}},0\right) \\
\left(3,\frac{5 \pi }{4}\right)
& \cdots
& \left(\frac{\cos(3)}{\sqrt{2}},\frac{\cos (3)}{\sqrt{2}},-3 \sin (3)\right)
& \left(-\frac{\sin (3)}{\sqrt{2}},\frac{\sin(3)}{\sqrt{2}},0\right)
& \left(\frac{\cos(3)}{\sqrt{2}},\frac{\cos (3)}{\sqrt{2}},0\right) \\
\left(3,\frac{7 \pi }{4}\right)
& \cdots
& \left(-\frac{\cos(3)}{\sqrt{2}},\frac{\cos (3)}{\sqrt{2}},-3 \sin (3)\right)
& \left(-\frac{\sin (3)}{\sqrt{2}},-\frac{\sin(3)}{\sqrt{2}},0\right)
& \left(-\frac{\cos(3)}{\sqrt{2}},\frac{\cos (3)}{\sqrt{2}},0\right)\\\hline
\end{array}
\end{equation}

If we represent the multivariate Hermite interpolation polynomial we obtain:
\begin{center}
\begin{tabular}{lll}
\includegraphics[height=.3\textheight]{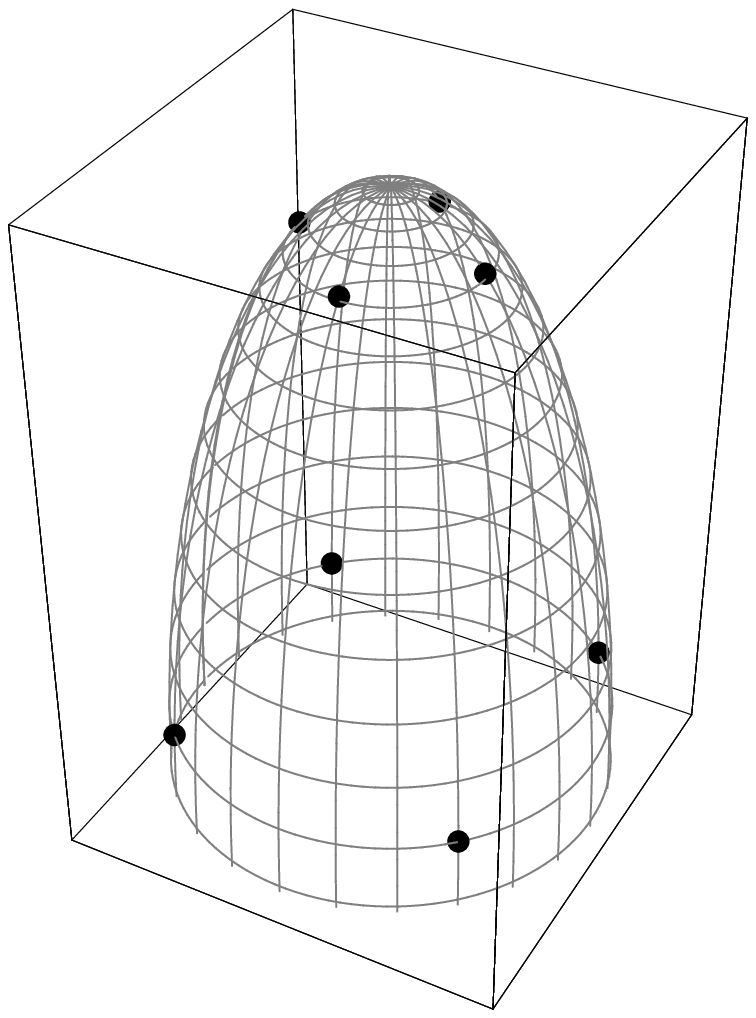}&
\hspace*{2cm}&
\includegraphics[height=.3\textheight]{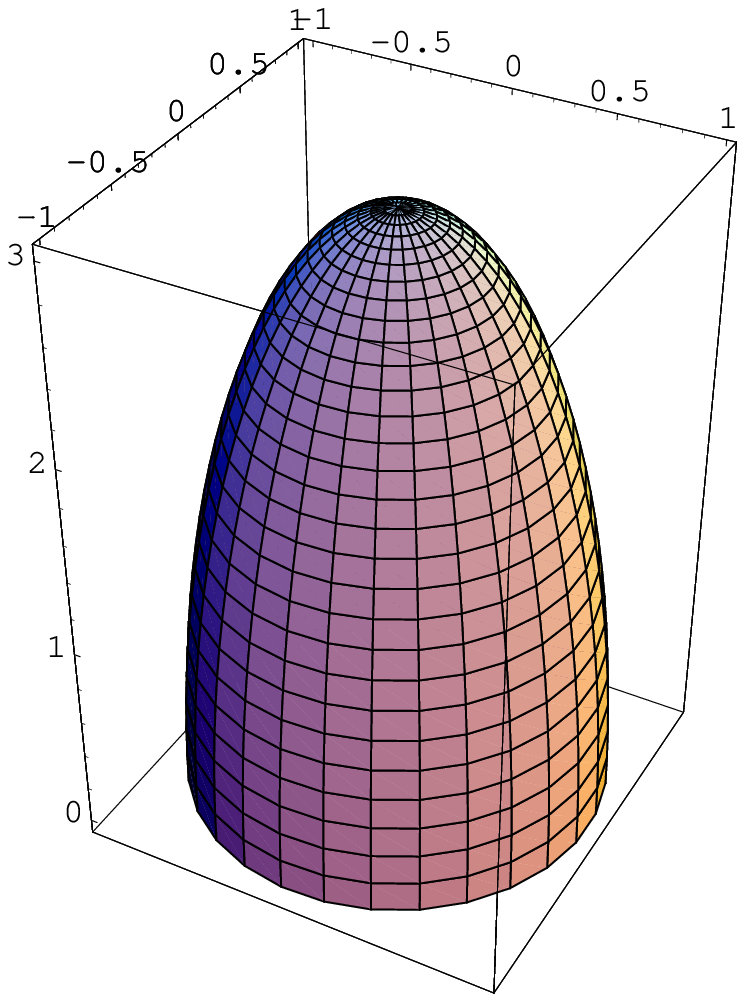}\\
\end{tabular}
\end{center}

This representation is a good approximation to the original function.

\begin{problem}
It would be of interest to measure the error produced by each polynomial interpolation
according to the order of the involved derivatives. Nevertheless we have chosen this
graphical approach to exemplify the goodness of the multivariate Hermite interpolation
method and postpone a deep study of errors to a forthcoming work.
\end{problem}

\begin{example}
Let us show a new example in which we increase the number of variables, from two to
four. We also use variables in different types: polynomial, logarithmic and fractional.
We are interested in studying the following function:
\[
F(X,Y,Z,T)=\frac{(-T+Z)^2\mathrm{Log}[1/X]}{Y}
\]
Observe that this function is of polynomial type in $T,Z$, of logarithmic type in $1/X$
and of fractional type in $X,Y$. We propose to study the error produced by the
interpolating polynomial when we consider a grid of points and evaluate the derivatives
at these points up to order 2, and show this error through the graphical representation
of sections fixing pairs of two variables.
\end{example}

Let us consider the following grid:
$\{(x,y,z,t)\in\mathbb{R}^4\mid\;{x,y,z,t}\in\{1,2,3\}\}$ and evaluate the functions:
\begin{multline*}
F, D_{(1,0,0,0)}F, D_{(1,1,0,0)}F, D_{(1,1,1,0)}F, D_{(1,1,1,1)}F,\\
D_{(1,0,0,0)}D_{(1,0,0,0)}F, D_{(1,1,0,0)}D_{(1,1,0,0)}F, D_{(1,1,0,0)}D_{(1,0,0,0)}F
\end{multline*}
in the grid. We call $P$ the interpolating polynomial.

In the following table we represent the sections of $F$ and $P$ for the following values
of $x\in[1,3],y=2,z\in[1,3],t=2$.

\begin{center}
\begin{tabular}{ccc}
\includegraphics[height=.18\textheight]{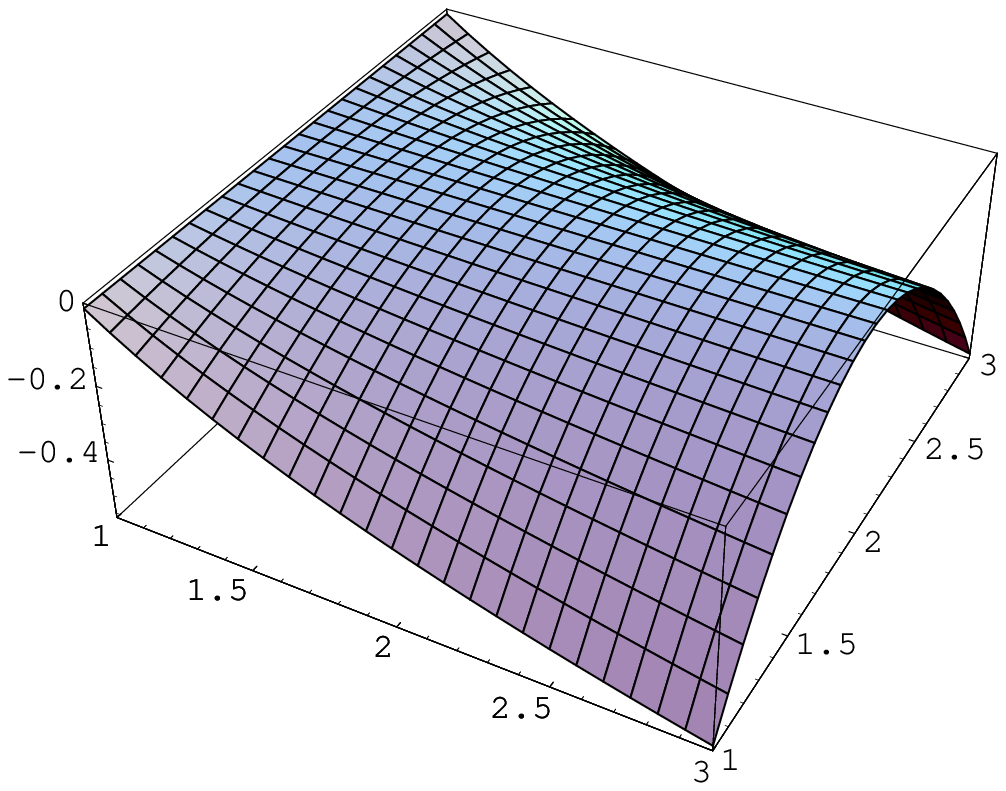}&
\includegraphics[height=.18\textheight]{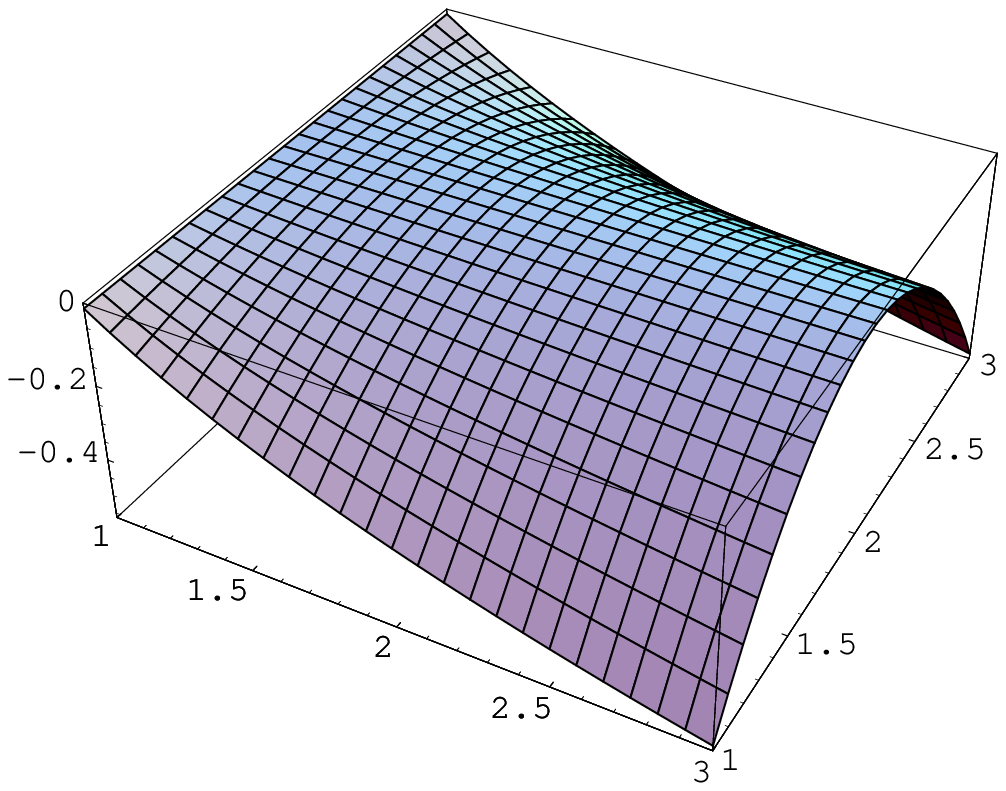}&
\includegraphics[height=.21\textheight]{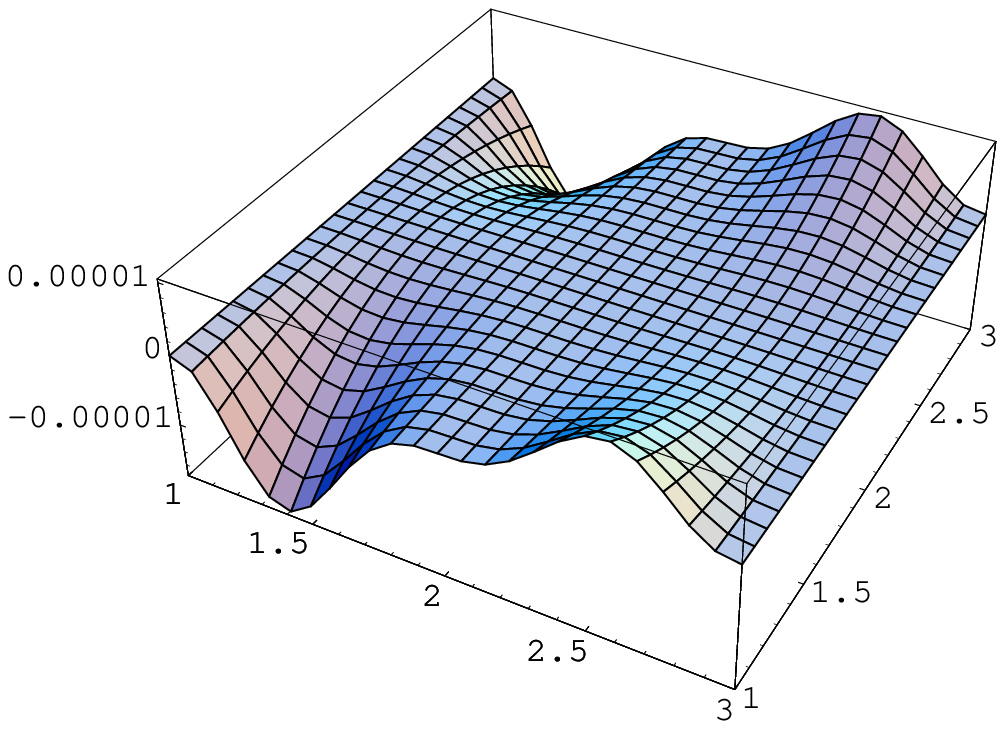}\\
F(X,Y,Z,T)&P(X,Y,Z,T)&$F-P$\\
\end{tabular}
\end{center}

In this case the maximum error is smaller than $0.000015$, and it is reached next to
$x=2.75$ and $z=1$ or 3. We have a similar behavior when we fix $x$ and $y\in[1,3]$; in
this case the maximum error is reached next to $y=1.4$. Hence we are interested in
fixing these two variables in these values. The representation of this section, for the
following values: $x=2.75,y=1.4,z,t\in[1,3]$.

\begin{center}
\begin{tabular}{ccc}
\includegraphics[height=.18\textheight]{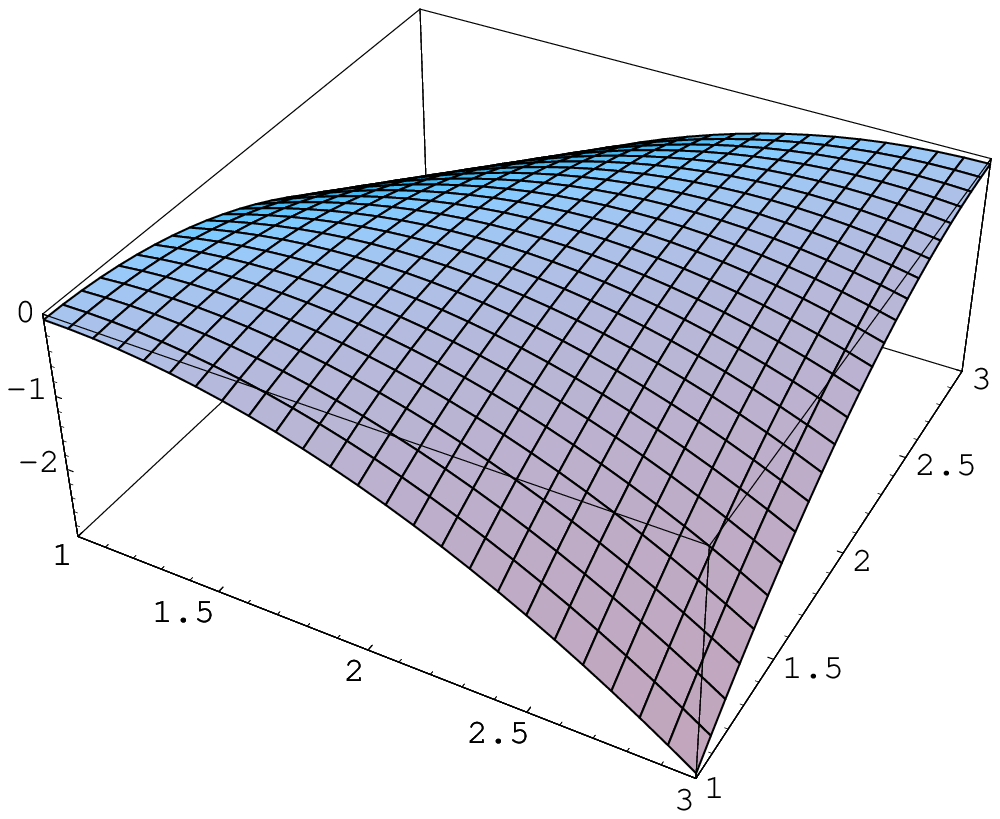}&
\includegraphics[height=.18\textheight]{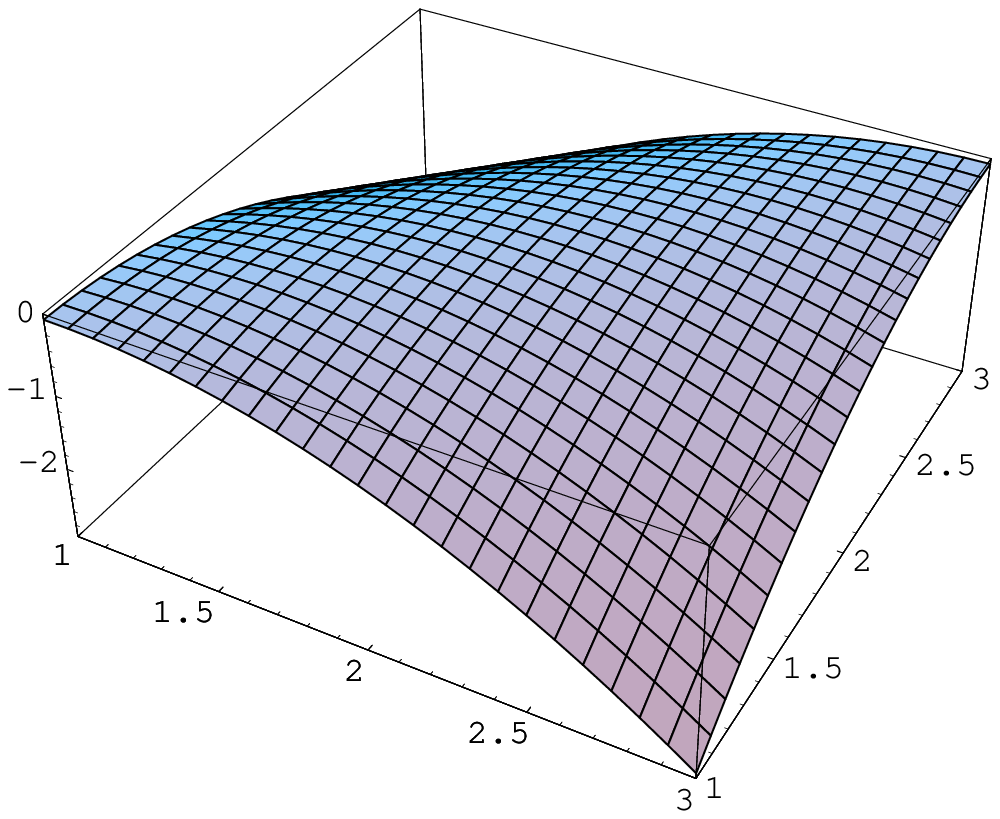}&
\includegraphics[height=.21\textheight]{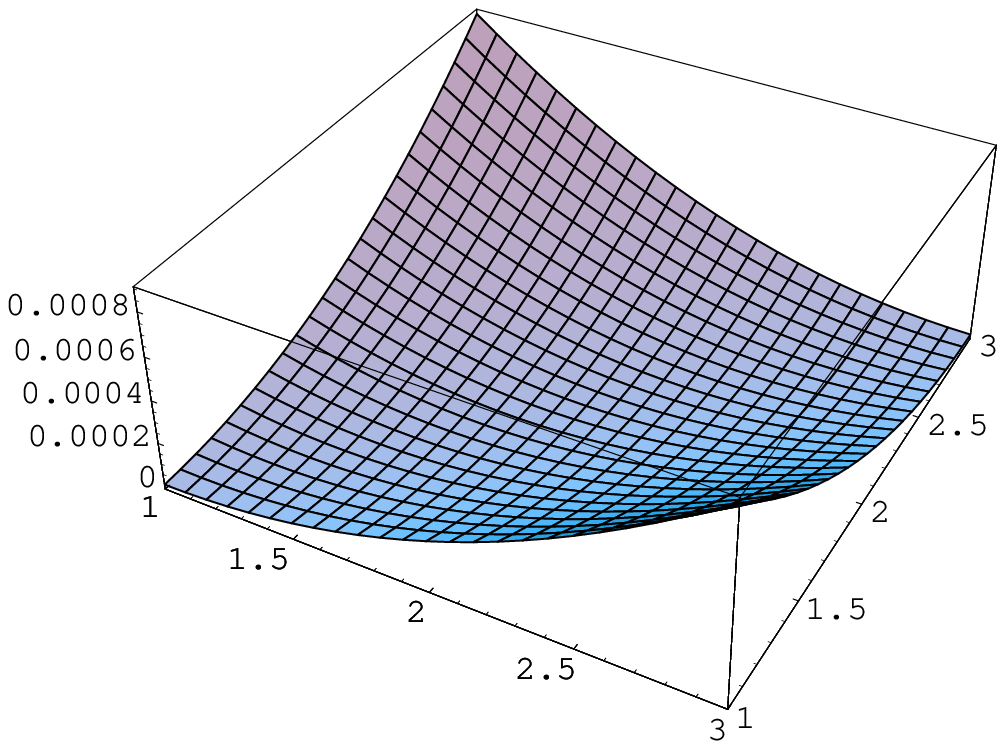}\\
F(X,Y,Z,T)&P(X,Y,Z,T)&$F-P$\\
\end{tabular}
\end{center}

In the case of $z,t$ the maximum error is reached next to $z,t=1$ or 3. If we fix these
values $z=1$, $t=3$ and allow $x,y$ to vary in $[1,3]$, we obtain the following
representation:

\begin{center}
\begin{tabular}{ccc}
\includegraphics[height=.18\textheight]{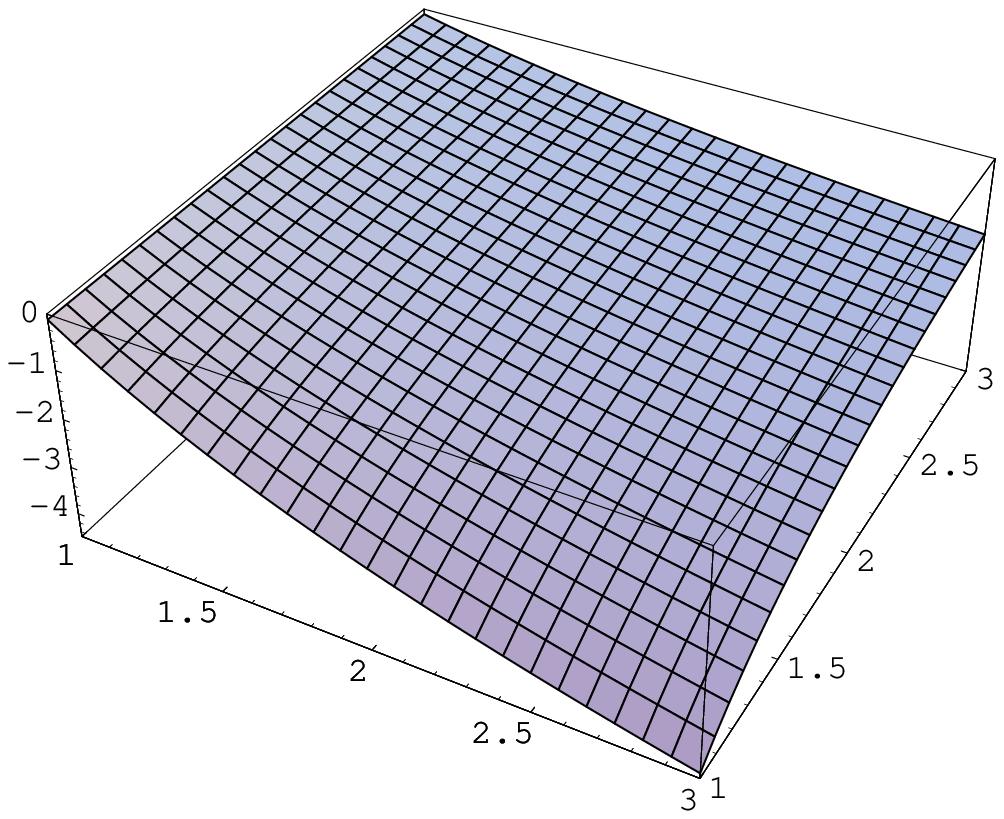}&
\includegraphics[height=.18\textheight]{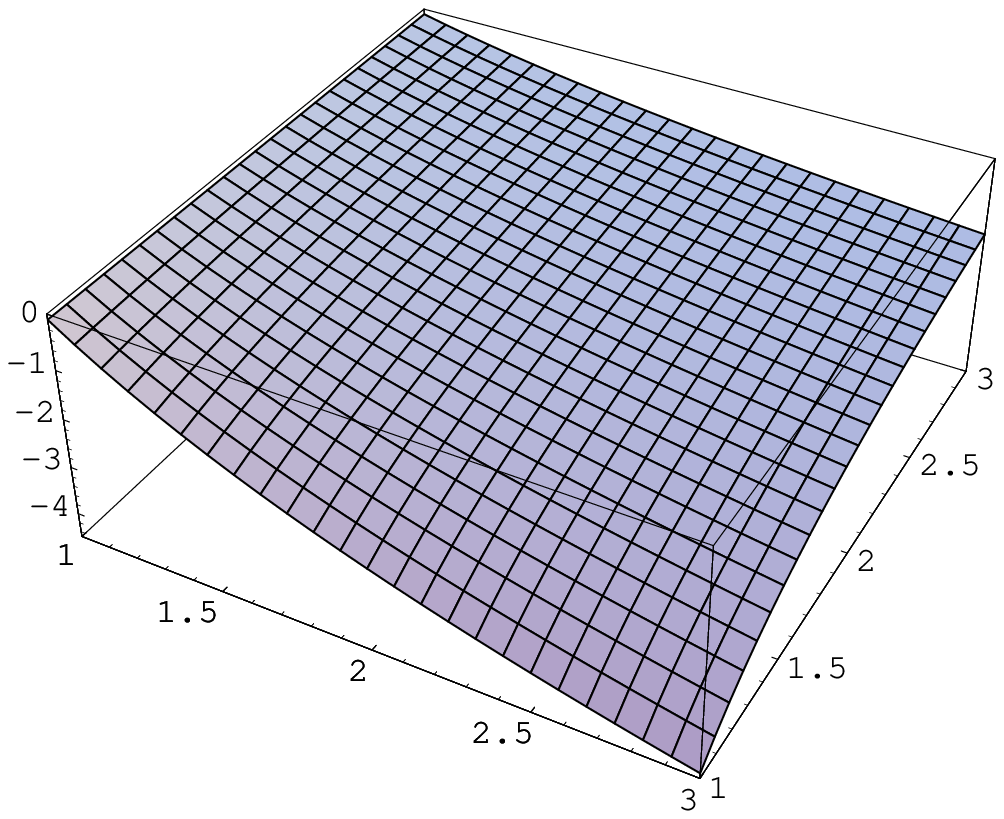}&
\includegraphics[height=.21\textheight]{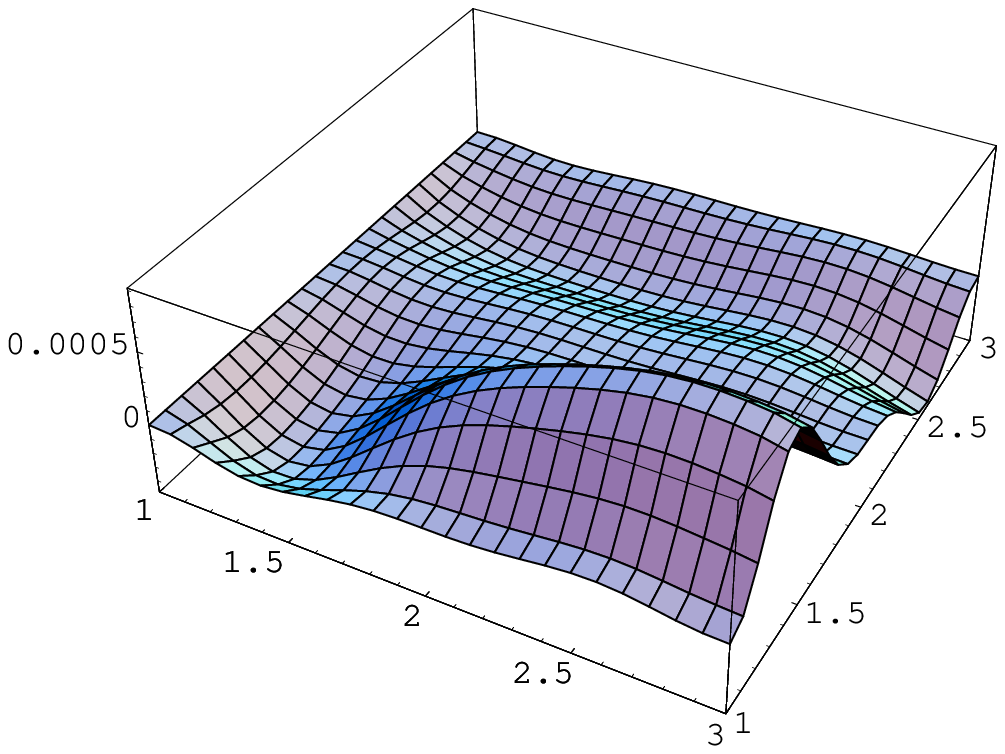}\\
F(X,Y,Z,T)&P(X,Y,Z,T)&$F-P$\\
\end{tabular}
\end{center}

In this case the maximum error is over $0.001$.

As a final example let us study a little variation of the former one in which we apply
the Birkhoff interpolating method.

\begin{example}
Take the real function $G(X,Y,Z)=F(X,Y,Z,2)$, the grid of points
$\{(x,y,z)\mid\;x,y,z\in\{1,2,3\}\}$ and evaluate the functions:
$$
G, D_{(1,1,0)}G, D_{(1,1,1)}G, D_{(1,0,0)}D_{(1,0,0)}G, D_{(1,0,0)}D_{(1,1,0)}G,
D_{(1,1,0)}D_{(1,1,0)}G.
$$
In order to determine a Birkhoff interpolating polynomial we compute the reduced
Groebner basis of the ideal $J$, see Section~\ref{se:bir}; this ideal has codimension
$27\times(1+3+6)=270$. Each gap in the Birkhoff conditions produces a parameter in the
interpolating polynomial; in this case we have $27\times4=108$ gaps which coincide with
the number of parameters. Hence the general Birkhoff interpolating polynomial is
determined as a solution of a system of 162 independent linear equations, whose
expression we omit because its large size.
\end{example}

\noindent{\authorA}. \direcA
\newline
\noindent{\authorB}. \direcB
\newline
\noindent{\authorC}. \direcC
\newline
\noindent{\authorD}. \direcD

\end{document}